\documentclass[ a4paper, 11pt]{article} % for arXiv
\pdfoutput=1
\usepackage{amsmath, amsfonts, amscd, amssymb,  amsthm, mathrsfs, ascmac}

\usepackage{latexsym}
\usepackage{mathtools} % \coloneqq が使える
\usepackage{longtable, geometry}
\usepackage{empheq} % \empheq が使える

\usepackage[english]{babel}
\usepackage[utf8]{inputenc}
\usepackage[dvipdfmx]{graphicx}
\usepackage{bbm}
\usepackage{enumitem}
\usepackage{nicefrac}
\usepackage{physics}
\usepackage[subrefformat=parens]{subcaption}
\usepackage[svgnames,psnames]{xcolor}
\usepackage[colorlinks,citecolor=DarkGreen,linkcolor=FireBrick,linktocpage,unicode]{hyperref}
\usepackage{url, hyperref}

\usepackage{xparse}
\usepackage{braket} % \Set が使える

\usepackage{tikz} %図を描く
\usetikzlibrary{positioning, intersections, calc, arrows.meta,math} %tikzのlibrary
\usetikzlibrary{intersections,calc,arrows}
\usetikzlibrary{decorations.pathreplacing}
\hypersetup{urlcolor=blue}

\usepackage{authblk}
\usepackage{bm} % \bm (ボールド体) が使える

\newcommand{\R}{\mathbb{R}}

 \let\a\relax
\newcommand{\a}{\bm{a}}

\let\mc\relax
\newcommand{\mc}{\mathcal}

\newcommand{\ve}{\varepsilon}

\DeclareMathOperator{\Dom}{Dom}
\DeclareMathOperator{\Ker}{Ker}

\DeclareMathOperator{\dist}{dist}

\theoremstyle{definition}
\newtheorem{thm}{Theorem}[section]
\newtheorem{defi}[thm]{Definition}
\newtheorem{prop}[thm]{Proposition}
\newtheorem{lem}[thm]{Lemma}
\newtheorem{cor}[thm]{Corollary}

\newtheorem{ex}[thm]{Example}
\newtheorem{rem}[thm]{Remark}

\title{ \sf
An Inverse Theorem for the Perron--Frobenius Theorem
}

\date{}

\author[1]{Shunsuke Tomioka\thanks{tomioka.shunsuke.t6@elms.hokudai.ac.jp}}

\affil[1]{Department of Mathematics,  Hokkaido University

Sapporo 060-0810,  Japan}

\begin{document}

\maketitle

\begin{abstract}
The Perron--Frobenius theorem in infinite-dimensional Hilbert spaces can be breifly stated as follows:
Given a Hilbert cone in a real Hilbert space, a bounded positive self-adjoint operator $A$ is ergodic with respect to this cone if and only if the maximum eigenvalue $\|A\|$ of $A$ is simple, and the corresponding eigenvector is strictly positive with respect to this cone.
This paper addresses the inverse problem of the Perron--Frobenius theorem: Does there exist a Hilbert cone such that a given bounded positive self-adjoint operator $A$ becomes ergodic when its maximum eigenvalue $\|A\|$ is simple? We provide an affirmative answer to this question in this paper. Furthermore, we conduct a detailed analysis of a specialized Hilbert cone introduced to obtain this result. Additionally, we provide an illustrative example of an application of the obtained results to the heat semigroup generated by the magnetic Schr\"{o}dinger operator.\\

\end{abstract}

\section{Introduction}

The Perron--Frobenius theorem is well-known as a theorem stating a sufficient condition for a real square matrix to have a simple maximum eigenvalue and finds widespread application across various domains. Originally proven by Perron in 1907 \cite{Perron1907} for positive matrices, it was later generalized by Frobenius in 1912 \cite{frobenius1912} to the case of irreducible non-negative matrices, hence bearing the names of both mathematicians. The first generalization of this theorem to the infinite-dimensional case was accomplished by Jentzsch in 1912 \cite{Jentzsch1912}.  Furthermore, a more general formulation in terms of convex cones in Banach spaces was provided by Krein and Rutman in 1948 \cite{krein1948}.
\medskip

Subsequently, the Perron--Frobenius theorem has been extended in various forms to operators on infinite-dimensional Hilbert spaces, continuously attracting the interest of many mathematicians \cite{Bonsall1958, Faris1972, Nussbaum1998, Schaefer1974}.
In diverse fields utilizing functional analysis such as mathematical physics and nonlinear differential equations theory, numerous applications have been made employing these extended versions of the Perron--Frobenius theorem. Examples include \cite{deimling1985, GREINER, Mallet2010, Miyao2019-2, MIYAO2021, Reed1978}. For further references, one may consult the citations within these works.
\medskip

In this paper, we adopt the framework provided by Faris for the extension of the Perron--Frobenius theorem to infinite-dimensional Hilbert spaces \cite{Faris1972}. The results from \cite{Faris1972} can be briefly summarized as follows (the precise statement will be given in the subsequent section):
Given a real Hilbert space $\mathcal{H}$ with a Hilbert cone $\mathcal{C}$, the following equivalence holds: a bounded positive self-adjoint operator $A$ on $\mathcal{H}$ is ergodic with respect to $\mathcal{C}$ if and only if the maximum eigenvalue $\|A\|$ of $A$ is simple, and the corresponding eigenvector is strictly positive with respect to $\mathcal{C}$.
\medskip

This paper investigates the inverse problem of the Perron--Frobenius theorem as described above. Specifically, we examine the question of whether there exists a Hilbert cone such that a given bounded positive self-adjoint operator $A$ becomes ergodic when its maximum eigenvalue $\|A\|$ is simple. We then provide an affirmative answer to this question. Furthermore, we delve into the properties of a particular class of Hilbert cones introduced to address this question.
\medskip

The structure of this paper is as follows:
In Section \ref{Sec2}, we provide definitions and properties of fundamental terms: positivity preservation, positivity improvement, and ergodicity. We also present the Perron--Frobenius theorem for bounded operators on infinite-dimensional Hilbert spaces.
Section \ref{Sec3} outlines the main results of this paper. We establish an   affirmative answer  to the inverse problem of the Perron--Frobenius theorem (Theorem \ref{InvPF}). 
Additionally, within this section, we delve into several considerations regarding a specific class of Hilbert cones introduced to derive these answers.
In Section \ref{Sec4}, we delve into various applications of the results obtained in Section \ref{Sec3} to heat semigroups. Furthermore, we discuss an illustrative example of an application concerning the heat semigroup generated by the magnetic Schr\"{o}dinger operator.
Appendix \ref{SecApp} covers fundamental aspects of self-adjoint operators on real Hilbert spaces. Given a real Hilbert space, we consider its complexification. In doing so, a natural extension of operators on the real Hilbert space to operators on the complexified Hilbert space is defined. Through this correspondence, various properties of operators on real Hilbert spaces naturally correspond to properties of operators on the complexified Hilbert space.  As comprehensive literature specifically addressing this topic seems limited, we compile relevant properties related to the content of this paper for the sake of completeness.

\subsection*{Acknowledgements}
I am deeply grateful to Professor Tadahiro Miyao for his guidance throughout this research and his assistance in proofreading the English manuscript.
I also thank him for including Example \ref{ex} as an example of an application of the main result of this paper. 

\vspace{2mm}

\noindent
{\bf Data Availability}\\
 Data sharing not applicable to this article as no datasets were generated or analysed during
the current study.
\vspace{2mm}

\noindent
{\bf Financial interests}\\
 The authors declare they have no financial interests.

\section{Preliminaries}\label{Sec2}
\subsection{Hilbert Cones}

In this subsection, we shall elucidate the fundamental aspects regarding Hilbert cones.

\begin{defi}\label{DefHC}
Let $\mathcal{H}$ denote a real Hilbert space. A non-empty subset $\mathcal{C}$ of $\mathcal{H}$ is termed a \textit{Hilbert cone} if it satisfies the following conditions (i)-(iv):
\begin{enumerate}
\item $\mathcal{C}$ is a convex cone in $\mathcal{H}$, i.e., $\mathcal{C} + \mathcal{C} \subset \mathcal{C}$, $[0, \infty) \cdot \mathcal{C} \subset \mathcal{C}$, and $\mathcal{C} \cap (-\mathcal{C}) = \{ 0 \}$.
\item $\mathcal{C}$ is a closed subset of $\mathcal{H}$.
\item For any $u, v \in \mathcal{C}$, it holds that $\langle u, v \rangle \geq 0$.
\item For any $w \in \mathcal{H}$, there exist $u, v \in \mathcal{C}$ such that $w = u - v$ and $\langle u, v \rangle = 0$.
\end{enumerate}
\end{defi}

\begin{defi}\label{StrictlyPositiveVector}
Let $\mathcal{H}$ denote a real Hilbert space and $\mathcal{C}$ denote a Hilbert cone.
A vector $u \in \mathcal{C} \setminus \{ 0 \}$ is termed a \textit{strictly positive vector with respect to $\mathcal{C}$} if for any $v \in \mathcal{C} \setminus \{ 0 \}$, it holds that $\langle v, u \rangle > 0$.
The set consisting of all strictly positive vectors with respect to $\mathcal{C}$ is denoted by $\mathcal{C}_{>0}$: $
\mathcal{C}_{>0} \coloneqq \{\, u \in \mathcal{C} \setminus \{ 0 \} \mid \langle v, u \rangle > 0 \ \mbox{for all $v \in \mathcal{C} \setminus \{ 0 \}$}\}.$
\end{defi}

Below, we demonstrate that the Hilbert cone can be characterized as a self-dual cone.

\begin{defi}
Let $\mathcal{H}$ denote a real Hilbert space.
For any subset $\mathcal{M}$ of $\mathcal{H}$, define $\mathcal{M}^{\dagger} \coloneqq \{\, v \in \mathcal{H} \mid \langle v, u \rangle \geq 0\ \text{for all } u \in \mathcal{M} \}$.
\end{defi}

\begin{prop}\it 
Let $\mathcal{H}$ be a real Hilbert space and $\mathcal{C}$ be a non-empty convex cone in $\mathcal{H}$.
Then, the condition $\mathcal{C}^{\dagger} = \mathcal{C}$ serves as both necessary and sufficient for $\mathcal{C}$ to be a Hilbert cone of $\mathcal{H}$.
\end{prop}
\begin{proof}
In the context of complex Hilbert spaces, the assertion of this proposition is well-known \cite{Bratteli1987, Miyao2019}. For the convenience of the reader, we provide the proof of the statement in the context of real Hilbert spaces here.

If $\mathcal{C}$ is a Hilbert cone, then by condition (iii) of Definition \ref{DefHC}, we have $\mathcal{C} \subset \mathcal{C}^{\dagger}$. Moreover, for any $w \in \mathcal{C}^{\dagger}$, there exist $u, v \in \mathcal{C}$ such that $w = u - v$ and $\langle u, v \rangle = 0$ by condition (iv) of Definition \ref{DefHC}. Hence, from $0 \leq \langle w, v \rangle = - \langle v, v \rangle$, we conclude that $v=0$, i.e., $w \in \mathcal{C}$.

If $\mathcal{C}$ is self-dual, we now demonstrate that it satisfies conditions (ii), (iii), and (iv) of Definition \ref{DefHC}.
Given that $\mathcal{C}^{\dagger} = \mathcal{C}$, conditions (ii) and (iii) hold true due to the closedness of $\mathcal{C}^{\dagger}$ and the definition of $\mathcal{C}^{\dagger}$. To establish condition (iv), let $w \in \mathcal{H}$ be arbitrary. Since $\mathcal{C}$ is a closed convex set, by the nearest point theorem, there exists $u \in \mathcal{C}$ such that $\|w - u\| = \text{dist}(w, \mathcal{C})$. Taking $v = u - w$, for any $t > 0$ and any $u' \in \mathcal{C}$, we have $u + tu' \in \mathcal{C}$, hence
\begin{align*}
\|w - u\|^2 \leq \|w - (u + tu')\|^2 = \|w - u\|^2 - 2t \langle w - u, u'\rangle + t^2 \|u'\|^2.
\end{align*}
This simplifies to $- \frac{t}{2} \|u'\|^2 \leq \langle v, u' \rangle$. Taking $t \to 0$, we obtain $0 \leq \langle v, u' \rangle$, implying $v \in \mathcal{C}^{\dagger} = \mathcal{C}$. Additionally, for any $-1 < t < 0$, we have $(1 + t)u \in \mathcal{C}$, hence
\begin{align*}
\|w - u\|^2 \leq \|w - (1 + t)u\|^2 = \|w - u\|^2 - 2t \langle w - u, u\rangle + t^2 \|u\|^2.
\end{align*}
This simplifies to $\langle v, u\rangle \leq - \frac{t}{2} \|u\|^2$. Taking $t \to 0$, we obtain $\langle v, u\rangle \leq 0$. Combining this with $\langle v, u\rangle \geq 0$, we conclude $\langle v, u\rangle = 0$.
\end{proof}

	\subsection{Positivity Preserving Operators, Positivity Improving Operators, Ergodicity}

In this subsection, we discuss the properties of operators determined by a given Hilbert cone. These concepts are employed to formulate Perron--Frobenius theorem in the case of infinite-dimensional real Hilbert spaces.

\begin{defi}
Let $\mathcal{H}$ denote a real Hilbert space, $\mathcal{C}$ be a Hilbert cone in $\mathcal{H}$, and $A \in \mathcal{B}(\mathcal{H})$. Here, $\mathcal{B}(\mathcal{H})$ represents the set of bounded linear operators on $\mathcal{H}$.
\begin{enumerate}
    \item If $A\mathcal{C} \subset \mathcal{C}$ holds true, $A$ is termed a \textit{positivity preserving operator  with respect to $\mathcal{C}$}.
    \item If $A(\mathcal{C} \setminus \{ 0 \}) \subset \mathcal{C}_{>0}$ holds true, $A$ is termed a \textit{positivity improving operator with respect to $\mathcal{C}$}.
\end{enumerate}
\end{defi}

\begin{defi}\label{ergodicity}
Let $\mathcal{H}$ be a real Hilbert space and $\mathcal{C}$ be a Hilbert cone in $\mathcal{H}$.

\begin{enumerate}
\item A subset $\mathcal{E}$ of $\mathcal{B}(\mathcal{H})$ is said to be \textit{ergodic with respect to $\mathcal{C}$} if it satisfies the following condition: for any $u, v \in \mathcal{C} \setminus \{ 0 \}$, there exists $A \in \mathcal{E}$ such that $\langle u, Av \rangle > 0$.

\item An operator $A \in \mathcal{B}(\mathcal{H})$ is said to be \textit{ergodic with respect to $\mathcal{C}$} if the sequence $\{ A^n \}_{n \in \mathbb{N}}$ is ergodic with respect to $\mathcal{C}$. Here, $\mathbb{N}$ denotes the set of natural numbers.
\end{enumerate}
\end{defi}

Positivity improving operators are ergodic. More generally, if the weak closure $\overline{\mathcal{E}}^w$ of a subset $\mathcal{E} \subset \mathcal{B}(\mathcal{H})$ contains positivity  improving operators, then $\mathcal{E}$ is ergodic. Indeed, for any positivity  improving operator $A \in \overline{\mathcal{E}}^w$, we can choose a net $(A_j)_{j \in J}$ in $\mathcal{E}$ such that $A_j \xrightarrow{w} A$. Accordingly, for any $u, v \in \mathcal{C} \setminus \{ 0 \}$, there exists $j_0 \in J$ such that $\langle u, A_j v \rangle > 0$ for all $j \geq j_0$.

\begin{thm}[Perron--Frobenius Theorem ({\cite[Theorem 1]{Faris1972}})]\label{Perron--Frobenius} \it
Let $\mathcal{H}$ be a real Hilbert space, $\mathcal{C}$ be a Hilbert cone in $\mathcal{H}$, and $A \in \mathcal{B}(\mathcal{H})$ be a positive self-adjoint operator that preserves positivity with respect to $\mathcal{C}$. Additionally, suppose $\|A\|$ is an eigenvalue of $A$. Then, the following statements {\rm (i)} and {\rm (ii)} are equivalent:
\begin{enumerate}
\item[\rm (i)] $A$ is ergodic with respect to $\mathcal{C}$.
\item[\rm (ii)] There exists a strictly  positive vector $w \in \mathcal{C}_{>0}$ such that $\Ker(\norm{A}I - A) = \mathbb{R}w$, where $I$ denotes the identity operator, $\Ker T$ denotes the kernel of a given operator $T$, and $\mathbb{R}$ represents the set of real numbers.
\end{enumerate}
\end{thm}

\begin{rem}
In particular, under the hypotheses of Theorem \ref{Perron--Frobenius}, if $A$ is positivity-improving, then the kernel of $\|A\|I - A$ is spanned by a unique strictly positive vector.
\end{rem}

In the subsequent section, we address the inverse problem of this theorem, namely, when a bounded positive self-adjoint operator $A$ possesses $\norm{A}$ as a simple eigenvalue, whether there exists a Hilbert cone such that $A$ preserves (improves) positivity with respect to that Hilbert cone.

\section{Main Results}\label{Sec3}

\subsection{One Solution to the Inverse Problem of the Perron--Frobenius Theorem}

In this subsection, we demonstrate that when the maximal eigenvalue of a bounded positive self-adjoint operator on a real Hilbert space is simple, the operator becomes a positivity  improving operator with respect to a Hilbert cone determined by the eigenvector corresponding to the maximal eigenvalue. We first define the Hilbert cone used in the main result of this paper:

\begin{defi}
Let $\mathcal{H}$ be a real Hilbert space, and $u_0$ be a unit vector in $\mathcal{H}$.
The cone $\mathcal{P}(u_0)$, defined as
\begin{align*}
				\mathcal{P}(u_0)
				\coloneqq
				\Set{ u \in \mc{H} | \langle u_0, u \rangle \geq \frac{\norm{u}}{\sqrt{2}} }
			\end{align*}
is called the \textit{cone with $u_0$ as its axis}.
\end{defi}

It can be readily verified that $\mathcal{P}(u_0)$ is a closed convex cone in $\mathcal{H}$. Furthermore, the following proposition holds:

\begin{prop}\it 
Let $\mathcal{H}$ be a real Hilbert space, and $u_0$ be a unit vector in $\mathcal{H}$.
Then, $\mathcal{P}(u_0)^{\dagger} = \mathcal{P}(u_0)$, implying that $\mathcal{P}(u_0)$ is a Hilbert cone in $\mathcal{H}$.
\end{prop}
\begin{proof}
Let $u, v \in \mathcal{P}(u_0)$. We perform an orthogonal decomposition of $u$ and $v$ with respect to $u_0$: $u = \langle u_0, u \rangle u_0 + u_1$ and $v = \langle u_0, v \rangle u_0 + v_1$. Then, by noting that $\langle u_0, u \rangle \geq \|u\|/\sqrt{2} \geq \|u_1\|$ and $\langle u_0, v \rangle \geq \|v\|/\sqrt{2} \geq \|v_1\|$, we have:
\[
\langle u, v \rangle = \langle u_0, u \rangle \langle u_0, v \rangle + \langle u_1, v_1 \rangle \geq \langle u_0, u \rangle \langle u_0, v \rangle - \|u_1\| \|v_1\| \geq 0.
\]
Thus, $\mathcal{P}(u_0) \subset \mathcal{P}(u_0)^{\dagger}$.

Next, we show that $\mathcal{P}(u_0)^{\dagger} \subset \mathcal{P}(u_0)$, or equivalently, $\mathcal{H} \setminus \mathcal{P}(u_0) \subset \mathcal{H} \setminus \mathcal{P}(u_0)^{\dagger}$. Let $u \in \mathcal{H} \setminus \mathcal{P}(u_0)$ be arbitrary.

If $\langle u_0, u \rangle < 0$, then $u \in \mathcal{H} \setminus \mathcal{P}(u_0)^{\dagger}$ trivially follows from the definition of $\mathcal{P}(u_0)^{\dagger}$.

If $\langle u_0, u \rangle \geq 0$, let $\displaystyle v = u_0 - \frac{u - \langle u_0, u \rangle u_0}{\|u - \langle u_0, u \rangle u_0\|}$. Then, $v \in \mathcal{P}(u_0)$ and
\[
\langle u, v \rangle = \langle u, u_0 \rangle - \|u - \langle u_0, u \rangle u_0\| < \frac{\|u\|}{\sqrt{2}} - \frac{\|u\|}{\sqrt{2}} = 0,
\]
because $\langle u_0, u \rangle < \|u\|/\sqrt{2}$. Thus, $u \in \mathcal{H} \setminus \mathcal{P}(u_0)^{\dagger}$.
\end{proof}

 Any strictly positive vector in $\mathcal{P}(u_0)$ is characterized as an interior point of $\mathcal{P}(u_0)$:

\begin{prop}\it 
Let $\mathcal{H}$ denote a real Hilbert space, and $u_0 \in \mathcal{H}$ be a unit vector. Then, $\mathcal{P}(u_0)_{>0} = \mathcal{P}(u_0)^{\circ}$. Here, $\mathcal{P}(u_0)^{\circ}$ represents the interior of $\mathcal{P}(u_0)$.
\end{prop}

\begin{proof}
To demonstrate $\mathcal{P}(u_0)^{\circ} \subset \mathcal{P}(u_0)_{>0}$, let $u \in \mathcal{P}(u_0)^{\circ}$ be arbitrarily chosen. For any $v \in \mathcal{P}(u_0) \setminus \{ 0 \}$, let us  consider the orthogonal decomposition of $u$ and $v$ with respect to $u_0$, denoted as $u = \langle u_0, u \rangle u_0 + u_1$ and $v = \langle u_0, v \rangle u_0 + v_1$ respectively. Then,
\begin{align*}
		\langle u_0, u \rangle& > \frac{\norm{u}}{\sqrt{2}} > \norm{u_1},
		\quad
		\langle u_0, v \rangle \geq \frac{\norm{v}}{\sqrt{2}} \geq \norm{v_1},\\
		\langle u, v \rangle
		&= \langle u_0, u \rangle \langle u_0, v \rangle + \langle u_1, v_1 \rangle
		\geq \langle u_0, u \rangle \langle u_0, v \rangle - \norm{u_1} \norm{v_1}.
	\end{align*}
If $\norm{v_1} = 0$, then $\langle u_0, v \rangle > 0$, implying $\langle u, v \rangle > 0$. If $\norm{v_1} > 0$, then $\langle u_0, u \rangle \langle u_0, v \rangle - \norm{u_1}\norm{ v_1} \geq (\langle u_0, u \rangle - \norm{u_1}) \norm{v_1} > 0$, thus $\langle u, v \rangle > 0$. Therefore, $u \in \mathcal{P}(u_0)_{>0}$.

Let $\partial \mathcal{P}(u_0)$ denote the boundary of $\mathcal{P}(u_0)$: $\partial \mathcal{P}(u_0) = \mathcal{P}(u_0) \setminus \mathcal{P}(u_0)^{\circ}$. For any $u \in \partial \mathcal{P}(u_0) \setminus \{ 0 \}$, let $u' \coloneqq 2 \langle u_0, u \rangle u_0 - u$, then
\begin{align*}
		\norm{u'}^2 &= 4 \langle u_0, u \rangle^2 - 4 \langle u_0, u \rangle^2 + \norm{u}^2 = \norm{u}^2  > 0,
		\\
		\langle u_0, u' \rangle &= 2 \langle u_0, u \rangle - \langle u_0, u \rangle = \langle u_0, u \rangle = \frac{\norm{u}}{\sqrt{2}} = \frac{\norm{u'}}{\sqrt{2}},
		\end{align*}
		which  implies that $u\rq{}\in  \mathcal{P}(u_0)$. From this, it follows that 
		$
		\langle u', u \rangle = 2 \langle u_0, u \rangle ^2 - \norm{u}^2 = 0.
	  $
Hence, $u \notin \mathcal{P}(u_0)_{>0}$. Consequently, we obtain $\mathcal{P}(u_0)_{>0} = \mathcal{P}(u_0)^{\circ}$.
\end{proof}

The following theorem constitutes one of the main results of this paper:

\begin{thm}\label{InvPF}\it 
Let $\mathcal{H}$ be a real Hilbert space, and $A \in \mathcal{B}(\mathcal{H})$ be a positive self-adjoint operator. Furthermore, assume that $\norm{A}$ is a simple eigenvalue of $A$, and let $u_0 \in \Ker(\norm{A} I - A)$ be a unit vector. Then, $A$ is a positivity improving  operator with respect to $\mathcal{P}(u_0)$.
\end{thm}

\begin{rem}
Geometrically, the above theorem implies that the shape of the image $A\mathcal{P}(u_0)$ under $A$ becomes ``sharper'' than the cone $\mathcal{P}(u_0)$.
\end{rem}

\begin{proof}[Proof of Theorem \ref{InvPF}]
For any $u \in \mathcal{P}(u_0) \setminus \{ 0 \}$, we have
\begin{align*}
		\langle u_0, Au \rangle
		= \langle Au_0, u \rangle
		= \norm{A} \langle u_0, u \rangle
		\geq \norm{A} \cdot \frac{\norm{u}}{\sqrt{2}}
		\geq \frac{\norm{Au}}{\sqrt{2}}.
	\end{align*}
Hence, $A$ preserves the positivity with respect to $\mathcal{P}(u_0)$.
If $u \in \Ker(\|A\|I - A)$, then in the above inequality, we have $\|A\| \langle u_0, u \rangle = \|Au\| > \frac{\|Au\|}{\sqrt{2}}$, which implies $Au \in \mathcal{P}(u_0)^{\circ}$.

On the other hand, $\|A\|\|u\| = \|Au\|$ is equivalent to $u \in \text{Ker}(\|A\|I - A)$. Hence, if $u \notin \text{Ker}(\|A\|I - A)$, then $\|A\|\|u\| > \|Au\|$, and thus $\langle u_0, Au \rangle > \frac{\|Au\|}{\sqrt{2}}$ holds.
\end{proof}

\subsection{Stability of Operator's Positivity Improvement and Ergodicity under Perturbation to Axis Vectors}

The following theorem asserts the stability of the operator's ergodicity concerning perturbation to axis vectors:

\begin{thm}\label{StaEr}\it 
Let $\mathcal{H}$ denote a real Hilbert space, and $A \in \mathcal{B}(\mathcal{H})$ represent a positive self-adjoint operator. Suppose that $\|A\|$ constitutes a simple eigenvalue of $A$, and let $u_0 \in \Ker(\|A\| I - A)$ be a unit vector. Then, for any unit vector $u_1 \in \mathcal{H}$ such that $\|u_1 - u_0\| < 1/\sqrt{2}$, $A$ is ergodic with respect to $\mathcal{P}(u_1)$.
\end{thm}
\begin{rem}
The geometric interpretation is as follows: First, recall from Theorem \ref{InvPF} that $A$ is ergodic with respect to $\mathcal{P}(u_0)$. As a consequence, after applying $A$ to vectors in  $\mathcal{H}$ multiple times, only the component in the direction of $u_0$ survives.  Therefore, if the axis vector $u_1$ is close to $u_0$, then since $u_0$ is contained in $\mathcal{P}(u_1)$, $A$ remains ergodic with respect to $\mathcal{P}(u_1)$ as well.
\end{rem}
\begin{proof}[Proof of Theorem \ref{StaEr}]
For any $u, v \in \mathcal{P}(u_1) \setminus \{ 0 \}$, we have
\[
\left\langle u,\,  \frac{A^{2n} v}{\|A\|^{2n}} \right\rangle
= \left\langle \frac{A^n u}{\|A\|^n},\  \frac{A^n v}{\|A\|^n} \right\rangle
\to \langle \langle u_0, u \rangle u_0, \langle u_0, v \rangle u_0 \rangle
= \langle u_0, u \rangle \langle u_0, v \rangle
\quad (\, n \to \infty \,).
\]
Hence, to demonstrate the ergodicity of $A$ with respect to  $\mathcal{P}(u_1)$, it suffices to prove $\langle u_0, u \rangle > 0$ for any $u \in \mathcal{P}(u_1) \setminus \{ 0 \}$. This is established by the following calculation:
	\begin{align*}
		\langle u_0, u \rangle
		= \langle u_0 - u_1, u \rangle + \langle u_1, u \rangle
		\geq - \norm{u_0 - u_1}\norm{u} + \frac{\norm{u}}{\sqrt{2}}
		= \left( \frac{1}{\sqrt{2}} - \norm{u_0 - u_1} \right) \norm{u}
		> 0.
	\end{align*}
\end{proof}

In Theorem \ref{InvPF}, $A$ is characterized as a positivity  improving operator with respect to $\mathcal{P}(u_0)$. However, in Theorem \ref{StaEr}, when perturbing the axis vector $u_0$, only a weaker property of ergodicity is retained, rather than the aforementioned positivity improvement. Assuming $A$ exhibits a spectral gap, the subsequent theorem affirms the preservation of a more robust property of positivity improvement of $A$ under perturbations to the axis vector $u_0$:

\begin{thm}\label{StaPI}\it 
Let $\mathcal{H}$ be a real Hilbert space, and $A \in \mathcal{B}(\mathcal{H})$ be a positive self-adjoint operator. Assume that $\|A\|$ is a simple eigenvalue of $A$, and  let $u_0 \in \Ker(\|A\| I - A)$ be a unit vector. Additionally, suppose:
\[
\lambda_1 \coloneqq \sup_{\substack{u \in (\Ker(\|A\| I - A))^{\perp} \\ \|u\| = 1}} \langle u, Au \rangle < \|A\|.
\]
Then, for any unit vector $u_1 \in \mathcal{H}$ with  $\|u_1 - u_0\| < r$, the operator  $A$ is  positivity-improving with respect to  $\mathcal{P}(u_1)$. Here,
\[
r\coloneqq \frac{1-\alpha^2}{4 \sqrt{2} (1+\alpha^2)}, \quad
\alpha \coloneqq \frac{\lambda_1}{\|A\|}.
\]
\end{thm}

\begin{rem}
From a geometric perspective, the presence of a spectral gap induces the \lq\lq{}upper half-plane" $\{\, u \in \mathcal{H} \mid \langle u_0, u \rangle > 0 \,\}$, with $u_0$ oriented upwards, to undergo a significant narrowing in the direction of $u_0$ upon mapping by $A$. Consequently, when the perturbation to the axis vector $u_0$ is sufficiently small or $u_1$ is sufficiently close to $u_0$, the image of $\mathcal{P}(u_1)$ under $A$ experiences a similar narrowing of its shape. This observation aligns with the notion that $A$ acts as a positivity  improving operator with respect to $\mathcal{P}(u_1)$.
\end{rem}

To prove Theorem \ref{StaPI}, we prepare two lemmas.

\begin{lem}\label{InqIn}\it 
Under the same assumptions as Theorem \ref{StaPI}, the following holds: For any $u \in \mathcal{H}$ and any $t \geq 1$, if $\displaystyle \langle u_0, u \rangle \geq \frac{1}{t} \|u\|$, then
\[
\langle u_1, Au \rangle \geq \left( \frac{1}{\sqrt{1 + \alpha^2(t^2 - 1)}} - \|u_1 - u_0\| \right) \|Au\|.
\]
\end{lem}
\begin{proof}
Let $u$ have an orthogonal decomposition with respect to $u_0$ as $u = \langle u_0, u \rangle u_0 + u'$. Then,
\begin{align*}
\norm{Au}^2
&= \norm{A}^2 \langle u_0, u \rangle^2 + \norm{Au'}^2 \\
&\leq \norm{A}^2 \langle u_0, u \rangle^2 + \lambda_1^2 \norm{u'}^2 \\
&= \norm{A}^2 \langle u_0, u \rangle^2 + \lambda_1^2 (\|u\|^2 - \langle u_0, u \rangle^2) \\
&\leq \left\{\norm{A}^2 + \lambda_1^2(t^2 - 1)\right\} \langle u_0, u \rangle^2.
\end{align*}
Hence, we have
\[
\langle u_0, Au \rangle = \norm{A} \langle u_0, u \rangle \geq \frac{\norm{A}}{\sqrt{\norm{A}^2 + \lambda_1^2(t^2 - 1)}} \|Au\| = \frac{1}{\sqrt{1 + \alpha^2(t^2 - 1)}} \|Au\|.
\]
Therefore,
\begin{align*}
\langle u_1, Au \rangle
&= \langle u_1-u_0, Au \rangle + \langle u_0, Au \rangle \\
&\geq \left( \frac{1}{\sqrt{1 + \alpha^2(t^2 - 1)}} - \|u_1 - u_0\| \right) \|Au\|.
\end{align*}
\end{proof}

The following elementary lemma can be proved easily:
\begin{lem}\label{ElInq}\it 
Let $c>0$. Consider the function $g$ defined on the interval $[0, \infty)$ by $g(x)\coloneqq x^4 -2 x^2 + c x - \frac{1}{4} $. Then, for $x<\min\{c/4, 1/(4c)\}$, it holds that $g(x)<0$.
\end{lem}

\begin{proof}[Proof of Theorem \ref{StaPI}]
Similar to the last part of the proof of Theorem \ref{StaEr}, we observe that:
\[
\mathcal{P}(u_1) \subset \left\{ u \in \mathcal{H} \ \middle|\ \langle u_0, u \rangle \geq \left( \frac{1}{\sqrt{2}} - \|u_1 - u_0\| \right) \|u\| \right\}.
\]
Choosing $t=(1/\sqrt{2} -\|u_1-u_0\|)^{-1}$ and applying Lemma \ref{InqIn}, it suffices for $A$ to be a positivity improving  operator with respect to  $\mathcal{P}(u_1)$ to satisfy the following inequality:
\begin{equation}\label{eq:inequality}
\frac{1}{\sqrt{1 + \alpha^2\left\{ \left( \frac{1}{\sqrt{2}} - \|u_1 - u_0\| \right)^{-2} - 1 \right\}}} - \|u_1 - u_0\| > \frac{1}{\sqrt{2}}.
\end{equation}
Let $g$ be the function given by Lemma \ref{ElInq}. By choosing $c=\sqrt{2}(1+\alpha^2)/(1-\alpha^2)$, it can be shown that inequality \eqref{eq:inequality} is equivalent to $g(\|u_1-u_0\|)<0$. Therefore, by applying Lemma \ref{ElInq}, the claim of Theorem \ref{StaPI} follows.
\end{proof}

\section{Application to Heat Semigroups}\label{Sec4}
In this section, we expound upon the application of the Hilbert cone $\mathcal{P}(u_0)$ introduced in the previous section to the heat semigroup generated by lower semi-bounded self-adjoint operators. Hereafter, we denote the spectrum of a given linear operator $T$ by $\sigma(T)$. The following proposition is a direct consequence of Theorem \ref{InvPF}:

\begin{prop}\label{BaseProp}\it 
Let $\mathcal{H}$ denote a real Hilbert space, $T \colon \mathcal{H} \to \mathcal{H}$ be a lower semi-bounded self-adjoint operator. Assume that  $\mu \coloneqq \inf \sigma(T)$ is  the simple eigenvalue of $T$, and let $u_0 \in \Ker(T - \mu I)$ be a unit vector. Then, for all $s>0$, $e^{-s T}$ is a positivity improving  operator with respect to $\mathcal{P}(u_0)$.
\end{prop}

The aim of this section is to investigate the stability of the properties stated in Proposition \ref{BaseProp} under perturbations to the operator $T$.

In the following, $\Dom T$ denotes the domain of the given linear operator $T$.
The following theorem claims the stability of the ergodicity of $\{e^{-sT}\}_{s\ge 0}$ stated in Proposition \ref{BaseProp} under perturbations when the perturbed operator $T$ possesses a spectral gap.

\begin{thm}\label{StaSG}\it 
Let $\mathcal{H}$ be a real Hilbert space, $T \colon \mathcal{H} \to \mathcal{H}$ a lower semi-bounded self-adjoint operator. Suppose $\mu \coloneqq \inf \sigma(T)$ is a simple  eigenvalue of $T$, and let $u_0 \in \mathrm{Ker}(T - \mu I)$ be a unit vector. Let $\{S(\kappa)\}_{\kappa \in J}$ be a family of symmetric operators indexed by a one-dimensional open interval $J$ containing $0$, satisfying the following:
\begin{enumerate}
\item[\rm (i)] For any $\kappa \in J$, $\Dom T \subset \Dom S(\kappa)$.
\item[\rm (ii)] There exist continuous functions $a, b \colon J \to [0, \infty)$ such that for any $\kappa \in J$,
\begin{align*}
\norm{S(\kappa)u} \leq a(\kappa) \norm{Tu} + b(\kappa) \norm{u} \quad \text{for all } u \in \Dom(T).
\end{align*}
\item[\rm (iii)] $a(0) = b(0) = 0$.
\item[\rm (iv)] 
There exists a  $\kappa_0 > 0$ such that $(T(\kappa))_{\abs{\kappa} < \kappa_0}$ has a uniform spectral gap : 
	\begin{align*}
		\delta \coloneqq \inf_{\abs{\kappa} < \kappa_0} \dist\left(\mu(\kappa),\,  \sigma(T + s(\kappa) )\setminus \{\mu(\kappa)\} \right) > 0,
	\end{align*}
where $T(\kappa) \coloneqq T + S(\kappa)$.
\end{enumerate}
Fix $s_0 > 0$ and define 
\begin{align*}
\ve &\coloneqq \frac{\delta}{2}, \\
c(\kappa) &\coloneqq a(\kappa) + \frac{(\left| \mu \right| + \ve )a(\kappa) + b(\kappa)}{\ve} \quad (\kappa \in J), \\
\alpha &\coloneqq 1 - e^{-s_0 \delta}, \\
r &\coloneqq \frac{1-\alpha^2}{4 \sqrt{2} (1+\alpha^2)}.
\end{align*}
Then, for any $s \in (0,  s_0]$ and $\abs{\kappa} < \kappa_0$  satisfying
\begin{align*}
c(\kappa) < \frac{r \sqrt{1 - r^2/4}}{1 + r \sqrt{1 - r^2/4}},
\end{align*}
 the semigroup $ e^{-s T(\kappa)}$ is positivity-improving with respect to $\mathcal{P}(u_0)$.
\end{thm}
\begin{proof}
Define $C_{\ve} \coloneqq \{\, z \in \mathbb{C} \mid \abs{z - \mu} = \ve \,\}$.
When $c(\kappa) < 1$, since $a(\kappa) < 1$, by the Kato-Rellich theorem in real Hilbert spaces (Theorem \ref{KTthm}), $T(\kappa)$ is self-adjoint. Let $\tilde{\mathcal{H}}$ denote the complexification of $\mathcal{H}$, and denote the extension of $T(\kappa)$ onto $\tilde{\mathcal{H}}$ as $\tilde{T}(\kappa)$. According to (vi) of Proposition \ref{RCCorr}, $\tilde{T}(\kappa)$ is self-adjoint on $\tilde{\mathcal{H}}$. Furthermore, since $\{\, z \in \mathbb{C} \mid \abs{z - \mu} \leq \ve \,\} \subset \rho(\tilde{T}(\kappa))$,  the projection operator
\[
P_{\kappa} \coloneqq \frac{1}{2\pi i} \int_{C_{\ve}} (z I - \tilde{T}(\kappa))^{-1} \, dz
\]
can be defined, following \cite[Chapter XII]{Reed1978}. 
Note that according to (iv) of Proposition \ref{RCCorr}, $P_0$ is a projection onto $\mathrm{Ker}(\tilde{T} - \mu I) = \mathrm{Ker}(T - \mu I) + i \mathrm{Ker}(T - \mu I)$. Moreover, when $c(\kappa) < 1/2$, we have $\displaystyle \norm{P_{\kappa} - P_0} < \frac{c(\kappa)}{1 - c(\kappa)} < 1$, implying that $P_{\kappa}u_0  \neq 0$ is an eigenvector of $\tilde{T}(\kappa)$. 
The vector $P_{\kappa}u_0$ is decomposed into its real part $u_{\kappa}$ and imaginary part $v_{\kappa}$ as: $P_{\kappa}u_0=u_{\kappa}+i v_{\kappa}$. Consequently, since $\|P_{\kappa}u_0-u_0\|<1$, employing Proposition \ref{RCCorr} (iv) reveals that $u_{\kappa}$ is an eigenvector of $T(\kappa)$.
Furthermore,
\begin{align*}
\norm{u_{\kappa}}^2
&= \norm{u_{\kappa} - u_0}^2 + \norm{u_0}^2 + 2 \langle u_{\kappa}- u_0, u_0 \rangle \\
&= \norm{u_{\kappa} - u_0}^2 - 1 + 2 \langle u_{\kappa}, u_0 \rangle.
\end{align*}
Thus, by using the elementary inequality $2\sqrt{ab}\le a+b\ (a\ge 0, b\ge 0)$, we get
\begin{align*}
\frac{\langle u_{\kappa}, u_0 \rangle}{\norm{u_{\kappa}}}
= \frac{1}{2} \left(\, \norm{u_{\kappa}} + \frac{1 - \norm{u_{\kappa} - u_0}^2}{\norm{u_{\kappa}}} \,\right) \geq \sqrt{1 - \norm{u_{\kappa} - u_0}^2}.
\end{align*}
Therefore, if we define $\displaystyle v_{\kappa} \coloneqq \frac{u_{\kappa}}{\norm{u_{\kappa}}}$, then
\begin{align*}
\norm{ v_{\kappa} - u_0 }^2
&= \norm{v_{\kappa}}^2 + \norm{u_0}^2 - 2 \langle v_{\kappa}, u_0 \rangle = 2( 1 - \langle v_{\kappa}, u_0 \rangle ) \\
&\leq 2 \left( 1 - \sqrt{1 - \norm{u_{\kappa} - u_0}^2} \right) \\
&\leq 2 \left\{ 1 - \sqrt{1 - \left( \frac{c(\kappa)}{1 - c(\kappa)} \right)^2} \right\}.
\end{align*}
Hence, if $\displaystyle c(\kappa) < \frac{r \sqrt{1 - r^2/4}}{1 + r \sqrt{1 - r^2/4}}$, then $\norm{ v_{\kappa} - u_0 } < r$. This, together with Theorem \ref{StaPI}, concludes that $e^{-s T(\kappa)}$ is a positivity  improving operator for any $s \in (0, s_0]$.
\end{proof}

\begin{rem}

As illustrated in the corollary below, Theorem \ref{StaSG} essentially indicates that \(e^{-sT(\kappa)}\) is positivity-improving when \(s>0\) is sufficiently small. In  the context  of functional analysis, the application of the Perron--Frobenius theorem to heat semigroups typically involves demonstrating that the semigroup in question is positivity-preserving  for all \(s \ge 0\) and ergodic, thereby establishing the uniqueness of the lowest eigenvalue of its generator (see, e.g., \cite{Faris1972}). Note, however, that Theorem \ref{StaSG} does not address whether \(e^{-sT(\kappa)}\) is positivity-preserving for \(s > s_0\). Although the statement of the theorem is somewhat limited in that sense, it is nevertheless sufficient to conclude that the lowest eigenvalue of \(T(\kappa)\) is simple and that the corresponding eigenvector is strictly positive with respect to \(\mathcal{P}(u_0)\).

\end{rem}

\begin{cor}\label{StaSG2}\it 
Let $\mathcal{H}$ be a real Hilbert space, and let $T \colon \mathcal{H} \to \mathcal{H}$ be a lower semi-bounded self-adjoint operator. Assume that $\mu \coloneqq \inf \sigma(T)$ is a simple eigenvalue of $T$, and let $u_0 \in \mathrm{Ker}(T - \mu I)$ be a unit vector. 
Let $S \colon \mathcal{H} \to \mathcal{H}$ be a symmetric operator that is relatively $T$-bounded, meaning there exist $a, b \geq 0$ such that
\[
\norm{Su} \leq a \norm{Tu} + b \norm{u}
\]
for all $u \in \Dom T$.
Moreover, we also assume that there exists a  $\kappa_0 > 0$ such that $(T + \kappa S)_{\abs{\kappa} < \kappa_0}$ has a uniform spectral gap : 
	\begin{align*}
		\delta \coloneqq \inf_{\abs{\kappa} < \kappa_0} \dist\left(\mu(\kappa), \sigma(T + \kappa S)\setminus \{\mu(\kappa)\} \right) > 0,
	\end{align*}
where $\mu(\kappa)$ is the lowest eigenvalue of $T + \kappa S$.

Then, 
for any fixed $s_0 > 0$, any $s \in (0, s_0]$ and any $\abs{\kappa} < \kappa_0$ satisfying
\[
	\abs{\kappa} < \left\{ a + \frac{(\abs{\mu} + \ve)a + b}{\ve} \right\}^{-1} \frac{r \sqrt{1 - r^2/4}}{1 + r \sqrt{1 - r^2/4}},
\]
where $\ve = \frac{\delta}{2}$ and $r = \frac{1-(1-e^{-s_0 \delta})^2}{4\sqrt{2}(1+(1-e^{-s_0\delta})^2)}$, the semigroup $e^{-s (T + \kappa S)}$ is positivity-improving with respect to $\mathcal{P}(u_0)$.

\end{cor}

\begin{proof}
First, for any $\kappa \in \mathbb{R}$, define
\[
S(\kappa) \coloneqq \kappa S,
\quad
a(\kappa) \coloneqq a \abs{\kappa},
\quad
b(\kappa) \coloneqq b \abs{\kappa}.
\]
Now, let $\alpha, c(\kappa)$ be defined as in Theorem \ref{StaSG}. Then, the condition $\displaystyle c(\kappa) < \frac{r \sqrt{1 - r^2/4}}{1 + r \sqrt{1 - r^2/4}}$ is equivalent to $\displaystyle \abs{\kappa} < \left\{ a + \frac{(\abs{\mu} + \ve)a + b}{\ve} \right\}^{-1} \frac{r \sqrt{1 - r^2/4}}{1 + r \sqrt{1 - r^2/4}}$. 
Thus, the assertion follows from Theorem \ref{StaSG}.
\end{proof}

\begin{rem}
Theorem \ref{StaSG} and Corollary \ref{StaSG2}  are consistent with results obtained through Kato's perturbation theory \cite{Kato1995, Reed1978}.
\end{rem}

\begin{ex}\label{ex}

We consider the following real Hilbert space:
\[
\mathcal{H}=\left\{
f\in L^2(\R^d)\, \Big|\, f^*(-x)=f(x)\ \mbox{a.e. $x$}
\right\}.
\]
Here, $L^2(\R^d)$ is the complex Hilbert space consisting of all complex-valued square-integrable functions on $\R^d$.
In this example, we propose an application of Theorem \ref{StaSG} using the {\it  magnetic Schr\"{o}dinger operator:}
\[
H=(-i \nabla+e\a)^2+V,
\]
where $e\in \R$ is a parameter.
Throughout this example, we assume the following:
\begin{itemize}
\item[(i)] $V$ is a multiplication operator by a {\it real-valued even} function $V(x)$ on $\R^d$, and there exist constants $a\in [0, 1)$ and $b\ge 0$ such that the following holds:
\[
\|Vf \|\le a \|(-\Delta)f \|+b \|f\|\quad \mbox{for all $f\in \Dom(-\Delta)$}.
\]
\item[(ii)] The vector potential $\a=(a_1, \dots, a_d)$ satisfies the following conditions: for all $j=1, \dots, d$,
$a_j$ is a multiplication operator by a {\it real-valued even} function $a_j(x)$ on $\R^d$, and
$
a_j(x)\in C_0^{\infty}(\R^d)
$.
\end{itemize}
Now, setting
\[
H_0=-\Delta+V,
\]
by assumption (i) and Kato--Rellich's theorem in real Hilbert spaces (Theorem \ref{KTthm}), $H_0$ is a lower semi-bounded self-adjoint operator on $\Dom(-\Delta)$. By defining
\[
H_{\rm I}(e)=e(-i\nabla) \cdot \a+e \a\cdot (-i \nabla)+e^2 \a^2,
\]
we have $H=H_0+H_{\rm I}(e)$, and it follows from assumptions (i) and (ii) that $H_{\rm I}(e)$ is infinitesimally small with respect to $H_0$. Moreover, it is straightforward to verify that $H$ maps $\mathcal{H}$ to $\mathcal{H}$. Our third assumption is as follows:
\begin{itemize}
\item[(iii)] $\mu\coloneqq\inf \sigma(H_0)$ is an eigenvalue of $H_0$, and $H_0$ has a spectral gap: $\displaystyle \inf_{\lambda\in \sigma(H_0)\setminus \{\mu\}}|\lambda-\mu|>0$.
\end{itemize}
We now consider the real Hilbert space:
\[
L^2_{\rm real}(\R^d)=\left\{f\in L^2(\R^d)\, \Big|\, f(x)\in \R \ \mbox{a.e. $x$}\right\}.
\]
The natural Hilbert cone in $L^2_{\rm real}(\R^d)$:
\[
\mathcal{C}=\left\{f\in L^2_{\rm real}(\R^d)\, \Big|\, \mbox{$f(x) \ge 0$ a.e. $x$}\right\}
\]
is often used in applications. Indeed, it is known that $e^{-s H_0}$ is a positivity improving operator with respect to $\mathcal{C}$ for all $s>0$. See, for example, \cite[Theorem XIII.46]{Reed1978}. On the other hand, the heat semigroup $\{e^{-s H}\}_{s\ge 0}$ generated by the magnetic Schr\"{o}dinger operator $H$ is {\it not} even ergodic with respect to $\mathcal{C}$. This fact can be immediately seen from the Feynman--Kac formula for $e^{-s H}$. See, for example, \cite{Simon2005}.\footnote{
Alternatively, this can be readily proven from the fact that $L^2_{\rm real}(\R^d)$ is not invariant under the action of $H_{\rm I}(e)$.
}

Now, combining the aforementioned facts with the Perron--Frobenius theorem for $\mathcal{C}$, we conclude that the lowest eigenvalue $\mu$ of $H_0$ viewed as an operator on $L^2_{\rm real}(\R^d)$ is simple. However, since the complexification of $L^2_{\rm real}(\R^d)$ is $L^2(\R^d)$, according to (iv) of Proposition \ref{RCCorr}, the simplicity of the lowest eigenvalue still holds when considering $H_0$ as an operator on $L^2(\R^d)$.

Next, let us  consider $H_0$ as an operator on $\mathcal{H}$. Since the complexification of $\mathcal{H}$ is also $L^2(\R^d)$, combining (iv) of Proposition \ref{RCCorr} with the previous analysis, we conclude that the lowest eigenvalue of $H_0$ viewed as an operator on $\mathcal{H}$ is simple. Therefore, denoting the corresponding normalized eigenvector as $\varphi$, by Proposition \ref{BaseProp} and Theorem \ref{StaSG}, there exists some $e_0>0$ such that for $|e|<e_0$, {\it both} heat semigroups $e^{-s H_0}$ and $e^{-s H}$ are positivity improving with respect to $\mathcal{P}(\varphi)$, provided  $s$ is small enough.

\end{ex}

\appendix

\section{Operators on Real Hilbert Spaces and Their Corresponding Operators on Complexified Hilbert Spaces}\label{SecApp}

The properties of operators on real Hilbert spaces can be explored in correspondence with operators extended to complexified Hilbert spaces, allowing for a discussion that closely parallels the case of operators on complex Hilbert spaces. While this correspondence may seem readily deducible, there appears to be a scarcity of readily accessible literature offering a complete demonstration.
Hence, in this appendix, for the convenience of the reader, we provide an overview of fundamental aspects concerning the aforementioned correspondence utilized in this paper.

Let $\mathcal{H}$ be a real Hilbert space, and let $T \colon \mathcal{H} \to \mathcal{H}$ be a densely defined operator. Then, the operator $T^* \colon \mathcal{H} \to \mathcal{H}$ defined as follows is called the adjoint operator of $T$:
\begin{align*}
    \mathrm{Dom}\, T^*
    &\coloneqq \left\{\, v \in \mathcal{H} \mid \text{There exists } w \in \mathcal{H} \text{ such that } \langle v, Tu \rangle = \langle w, u \rangle \text{ for all } u \in \mathrm{Dom}\, T \,\right\}, \\
    T^*v
    &\coloneqq w 
    \quad (\,  v \in \mathrm{Dom}\, T^* \,).
\end{align*}
If $T \subset T^*$, $T$ is called a symmetric operator; if $T^* = T$, $T$ is called a self-adjoint operator.

\begin{prop}\label{BasicA}\it 
Let $\mathcal{H}$ be a real Hilbert space, and let $T \colon \mathcal{H} \to \mathcal{H}$ be a densely defined operator. Then, $\mathrm{Ker}\, T^* = (\mathrm{Ran}\, T)^{\perp}$ holds. Here, $\mathrm{Ran}\, T$ denotes the range of $T$.
\end{prop}

\begin{proof}
First, we prove $\mathrm{Ker}\, T^* \subset (\mathrm{Ran}\, T)^{\perp}$. For any $u \in \mathrm{Ker}\, T^*$ and any $v \in \mathrm{Dom}\, T$, we have $\langle u, Tv \rangle = \langle T^*u, v \rangle = 0$. Thus, $u \in (\mathrm{Ran}\, T)^{\perp}$. Next, we prove $(\mathrm{Ran}\, T)^{\perp} \subset \mathrm{Ker}\, T^*$. For any $u \in (\mathrm{Ran}\, T)^{\perp}$ and any $v \in \mathrm{Dom}\, T$, we have $\langle u, Tv \rangle = 0$. Hence, $u \in \mathrm{Dom}\, T^*$ and $T^*u = 0$, which implies $u \in \mathrm{Ker}\, T^*$.
\end{proof}

\begin{prop}\label{SABI}\it 
	Let $\mathcal{H}$ be a real Hilbert space, and $T \colon \mathcal{H} \to \mathcal{H}$ be a symmetric operator.
	If there exists $\lambda \in \mathbb{R}$ such that $T - \lambda I$ is surjective, then $T$ is a self-adjoint operator.
\end{prop}
\begin{proof}
	From the surjectivity of $T - \lambda I$, for any $u \in \mathrm{Dom} \, T^*$, there exists $v \in \mathrm{Dom} \, T$ such that $(T^* - \lambda I)u = (T - \lambda I)v$.
	Hence, $(T^* - \lambda I)(u-v) = 0$.
	Since $T^* - \lambda I$ is injective by Proposition \ref{BasicA}, we have $u-v = 0$.
	Consequently, $\mathrm{Dom} \, T^* \subset \mathrm{Dom} \, T$, i.e., $T^* = T$.
\end{proof}

\begin{defi}
	Let $\mathcal{H}$ be a real Hilbert space, and $\tilde{\mathcal{H}}$ be its complexification. For a linear operator $T \colon \mathcal{H} \to \mathcal{H}$, we define a linear operator $\tilde{T} \colon \tilde{\mathcal{H}} \to \tilde{\mathcal{H}}$ as follows:
	\[
		\mathrm{Dom} \, \tilde{T} \coloneqq \mathrm{Dom} \, T + i \mathrm{Dom} \, T, \quad \tilde{T}(u+iv) \coloneqq Tu + iTv \quad (\,  u, v \in \mathrm{Dom} \, T \,).
	\]
\end{defi}

\begin{prop}\label{RCCorr}\it 
	Let $\mathcal{H}$ be a real Hilbert space, and $T \colon \mathcal{H} \to \mathcal{H}$ be a linear  operator. Then, the following statements hold:
	\begin{enumerate}
		\item[\rm (i)] $\mathrm{Dom} \, T$ is dense in $\mathcal{H}$ if and only if  $\mathrm{Dom} \, \tilde{T}$ is dense  in $\tilde{\mathcal{H}}$. Moreover,  $(\tilde{T})^* = (T^*)^{\sim}$ holds.
		
		\item[\rm (ii)] $\mathrm{Ran} \, T$  is dense in $\mathcal{H}$ if and only if  $\mathrm{Ran} \, \tilde{T}$ is dense in $\tilde{\mc{H}}$.
		
		\item[\rm (iii)] $T$ is  injective if and only if  $\tilde{T}$ is  injective.
		
		\item[\rm (iv)] For any $\lambda \in \mathbb{R}$, $\lambda$ is  an eigenvalue of $T$ if and only if  $\lambda$  is an eigenvalue of $\tilde{T}$. Moreover, $\mathrm{Ker} \, (\tilde{T} - \lambda I) = \mathrm{Ker} \, (T - \lambda I) + i \, \mathrm{Ker} \, (T - \lambda I)$ holds.
		
		\item[\rm (v)] $T$ is  bounded if and only if $\tilde{T}$ is  bounded. Moreover, $\| \tilde{T} \| = \| T \|$ holds.
		
		\item[\rm (vi)] $T$  is  symmetric (resp. self-adjoint) if and only if  $\tilde{T}$ is  symmetric (resp. self-adjoint).
		
		\item[\rm (vii)] Let $T$ be a symmetric operator and $\gamma \in \mathbb{R}$. Then, $T \geq \gamma I$ is equivalent to $\tilde{T} \geq \gamma I$.
	\end{enumerate}
\end{prop}

\begin{proof}
The first half of (i) and (ii) are immediate: 
In general, for any subspace $\mathcal{D}$ of $\mathcal{H}$, $\mathcal{D}$ being dense in $\mathcal{H}$ is equivalent to $\mathcal{D} + i \mathcal{D}$ being dense in $\tilde{\mathcal{H}}$.

To prove the latter half of (i),
consider any $u, v \in \Dom T^*$ and any $x, y \in \Dom T$.
\begin{align*}
\langle u+iv, \tilde{T}(x+iy) \rangle
&= \langle u+iv, Tx+iTy \rangle\\
&= \langle u, Tx \rangle + \langle v, Ty \rangle + i(\langle u, Ty \rangle - \langle v, Tx \rangle)\\
&= \langle T^*u, x \rangle + \langle T^*v, y \rangle + i(\langle T^*u, y \rangle - \langle T^*v, x \rangle)\\
&= \langle T^*u+iT^*v, x+iy \rangle\\
&= \langle (T^*)^{\sim}(u+iv), x+iy \rangle.
\end{align*}
Hence, $(T^*)^{\sim} \subset (\tilde{T})^*$.
Conversely, for any $u+iv \in \Dom (\tilde{T})^*$, let $x+iy \coloneqq (\tilde{T})^*(u+iv)$.
For any $w \in \Dom T$, we have
\begin{align*}
\langle x, w \rangle - i \langle y, w \rangle
=& \langle x+iy, w \rangle
= \langle (\tilde{T})^*(u+iv), w \rangle
= \langle u+iv, \tilde{T}w \rangle\\
= &\langle u+iv, Tw \rangle
= \langle u, Tw\rangle -i \langle v, Tw \rangle.
\end{align*}
Thus, $\langle x, w \rangle = \langle u, Tw\rangle$ and $\langle y, w \rangle = \langle v, Tw \rangle$.
As $w$ is arbitrary in $\Dom T$, $u,  v \in \Dom T^*$ and $T^*u = x, \ T^*v = y$.
Thus, we obtain $(\tilde{T})^* \subset (T^*)^{\sim}$.
Hence, we have shown $(\tilde{T})^* = (T^*)^{\sim}$.

(iii) If $\tilde{T}$ is injective, it is straightforward to show that $T$ is injective.
Conversely, assuming $T$ is injective, we need to show that $\tilde{T}$ is injective.
Suppose $\tilde{T}(u+iv) = 0$, i.e., $Tu + i Tv = 0$.
This implies $Tu = Tv = 0$.
Since $T$ is injective, $u=v=0$, hence $u+iv = 0$.

For the first part of (iv), we can apply (iii) to $T-\lambda I$.
To prove $\Ker(\tilde{T} - \lambda I) = \Ker(T - \lambda I) + i \Ker(T - \lambda I)$,
it is  immediate to see $\Ker(T - \lambda I) + i \Ker(T - \lambda I) \subset \Ker(\tilde{T} - \lambda I)$, so we prove $\Ker(\tilde{T} - \lambda I) \subset \Ker(T - \lambda I) + i \Ker(T - \lambda I)$.
For any $u + iv \in \Ker(\tilde{T} - \lambda I)$, we have $Tu + i Tv = \tilde{T}(u+iv) = \lambda(u+iv) = \lambda u + i \lambda v$.
Thus, $Tu = \lambda u, \ Tv = \lambda v$, implying $u, v \in \Ker(T - \lambda I)$.

For (v), if $T$ is bounded, for any $u, v \in \Dom T$,
\[
\norm{\tilde{T}(u+iv)}^2 = \norm{Tu + i Tv}^2=\norm{Tu}^2 + \norm{Tv}^2 \leq \norm{T}^2(\norm{u}^2 + \norm{v}^2) = \norm{T}^2 \norm{u+iv}^2.
\]
Thus, $\tilde{T}$ is bounded with $\norm{\tilde{T}} \leq \norm{T}$. Conversely,  if $\tilde{T}$ is bounded, then trivially $T$ is bounded, and $\norm{T} \leq \norm{\tilde{T}}$ holds.
Hence, we obtain $\left\|\tilde{T}\right\| = \norm{T}$.

(vi) is immediate from (i).

For (vii), if $\tilde{T} \geq \gamma I$, then $T \geq \gamma I$ is evident.
For the converse, assuming $T \geq \gamma I$, for any $u, v \in \Dom T$,
\begin{align*}
\langle u+iv, \tilde{T}(u+iv) \rangle
&= \langle u+iv, Tu+i Tv \rangle
= \langle u, Tu \rangle + \langle v, Tv \rangle\\
&\geq \gamma (\langle u, u \rangle + \langle v, v \rangle)
= \gamma \langle u+iv, u+iv \rangle.
\end{align*}
Thus, $\tilde{T} \geq \gamma I$.
\end{proof}

\begin{defi}\label{Real_Hilbert_Operator_Spectrum}
	Let $\mathcal{H}$ be a real Hilbert space, and $T \colon \mathcal{H} \to \mathcal{H}$ be a self-adjoint operator.
	Define the resolvent and spectrum of $T$ as those of $\tilde{T}$:
	$\rho(T) = \rho(\tilde{T})$ and $\sigma(T) = \sigma(\tilde{T})$.
\end{defi}

Proposition \ref{RCCorr} implies that various properties concerning the spectrum of self-adjoint operators on real Hilbert spaces completely correspond to those of the extended operators on complexified Hilbert spaces. Thus, calculations based on spectral decompositions are feasible for self-adjoint operators on real Hilbert spaces as well.

\begin{thm}[Kato--Rellich Theorem in Real Hilbert Spaces]\label{KTthm}\it 
Let $\mathcal{H}$ be a real Hilbert space, and let $T \colon \mathcal{H} \to \mathcal{H}$ be a lower semi-bounded self-adjoint operator. 
Moreover, let $S \colon \mathcal{H} \to \mathcal{H}$ be a symmetric operator which is $T$-bounded with a relative bound less than $1$. 
Then, $T+S$ is lower semi-bounded and self-adjoint.
\end{thm}

\begin{proof} 
The proof of this theorem closely resembles that of \cite[Theorem X. 12]{Reed1975}. For the reader's convenience, we provide a brief outline of the proof below.
By the inclusion $\Dom T \subset \Dom S$, it follows that $\Dom(T+S) = \Dom T$. Hence, $\Dom(T+S)$ is dense, and $T+S$ is symmetric. Also, by the assumption on $S$, there exist $a \in [0, 1)$ and $b \geq 0$ such that for all $u \in \Dom T$,
	$
		\norm{Su} \leq a \norm{Tu} + b \norm{u}
	$ holds.
Thus, for any sufficiently small $\lambda \in \R$ and any $u \in \mc{H}$,
	\begin{align*}
		\norm{S(T-\lambda I)^{-1}u}
		\leq& a \norm{T(T-\lambda I)^{-1}u} + b \norm{(T-\lambda I)^{-1}u}\\
		\leq& \left\{ a\norm{T(T-\lambda I)^{-1}} + b\norm{(T-\lambda I)^{-1}} \right\} \norm{u}.
	\end{align*}
Since $T$ is lower semi-bounded, there exists $\lambda_0$ such that for $\lambda<\lambda_0$,
$\norm{S(T-\lambda I)^{-1}} < 1$. Thus, $T+S - \lambda I = \{I + S(T-\lambda I)^{-1}\}(T-\lambda I)$ is bijective.
By Proposition \ref{SABI}, $T+S$ is self-adjoint. Also, since $\lambda\in \rho(T+S)$, $T+S$ is bounded from below.
\end{proof}

%\bibliographystyle{abbrvurl}

%\bibliography{Reference}

\end{document}